\documentclass[reqno,a4paper, 11pt]{amsart}

\usepackage[a4paper=true,pdfpagelabels]{hyperref}
\usepackage{graphicx}
\usepackage{mathabx}
\usepackage[ansinew]{inputenc}
\usepackage{amsfonts,epsfig}
\usepackage{latexsym}
\usepackage{amsmath}
\usepackage{amssymb}
\usepackage{color}

\newtheorem{theorem}{Theorem}

\newtheorem{corollary}[theorem]{Corollary}

\theoremstyle{definition}

\theoremstyle{remark}

\numberwithin{equation}{section}

\newcommand{\B}{\mathcal{B}}

\newcommand{\D}{\mathbb{D}}
\newcommand{\DD}{\widehat{\mathcal{D}}}

\newcommand{\DDD}{\mathcal{D}}
\newcommand{\N}{\mathbb{N}}
\newcommand{\RR}{\mathbb{R}}

\newcommand{\C}{\mathbb{C}}

\newcommand{\e}{\varepsilon}

\newcommand{\conz}{\overline{z}}
\newcommand{\cona}{\overline{a}}

\renewcommand{\phi}{\varphi}

\newcommand{\T}{\mathbb{T}}

\newcommand{\M}{\mathcal{M}}

\def\BMO{\mathord{\rm BMO}}

           \def\e{\varepsilon}
     \def\om{\omega}      
              
                  \def\z{\zeta}
                  
\def\G{\Gamma}

\renewcommand{\H}{\mathcal{H}}

\begin{document}

\title[Optimal off-diagonal estimates for Bergman kernels]{Optimal off-diagonal upper estimates for Bergman reproducing kernels}

\author[Atte Pennanen]{Atte Pennanen}
\address{University of Eastern Finland\\
Department of Physics and Mathematics\\
P.O.Box 111\\FI-80101 Joensuu\\
Finland}
\email{atte.pennanen@uef.fi}

\author[Jouni R\"atty\"a]{Jouni R\"atty\"a}
\address{University of Eastern Finland\\
Department of Physics and Mathematics\\
P.O.Box 111\\FI-80101 Joensuu\\
Finland}
\email{jouni.rattya@uef.fi}

\author[Siyu Wang]{Siyu Wang}
\address{Fudan University\\ School of Mathematical Sciences\\
Shanghai 200433\\ P.R.China}

\address{University of Eastern Finland\\
Department of Physics and Mathematics\\
P.O.Box 111\\FI-80101 Joensuu\\
Finland}
\email{siyuwang@fudan.edu.cn, siyuwang@uef.fi}

\author[Fanglei Wu]{Fanglei Wu}
\address{University of Eastern Finland\\
Department of Physics and Mathematics\\
P.O.Box 111\\FI-80101 Joensuu\\
Finland}
\email{fanglei.wu@uef.fi}


\thanks{All the authors are supported in part by Academy of Finland 356029. The first author is supported by Finnish Cultural Foundation, North Karelia Regional fund. The third author is supported by China Postdoctoral Science Foundation (No. 2023TQ0065, No. 2023M740716, No. GZB20230166) and Shanghai Post-doctoral Excellence Program (No. 2023186).}



\begin{abstract}
In this paper, we establish a sharp off-diagonal pointwise upper estimate for the Bergman reproducing kernel associated with a radial weight on the unit disc. Our proof is self-contained and assumes the weight satisfies a natural one-sided doubling condition on its moments. The standard kernels demonstrate that this estimate is sharp, up to a multiplicative constant, at every point. We also show that the estimate in fact characterizes the class of radial doubling weights under consideration.

As applications, we first obtain optimal $L^p$-mean estimates for certain modified Bergman kernels. This approach recovers the key estimates in [Pel\'aez~et~al., J.~Math.~Pures Appl. 105(2016), 102--130] and [Pel\'aez et al., arxiv.org/pdf/2407.04645] via a novel and more direct proof. Second, we establish novel connections between the non-tangential maximal function of the Berezin transform, the H\"ormander maximal function, and Carleson measures. Notably, these connections are new even in the setting of the standard weighted Bergman spaces. Finally, we extend our main results to higher dimensions, harmonic Bergman kernels, and two-weight fractional derivatives of kernels, the latter of which yields sharp estimates for Dirichlet reproducing kernels.
\end{abstract}

\maketitle

\section{Introduction and main results}

Understanding the behavior of Bergman reproducing kernels is fundamental because these kernels encode, in a very concrete way, both the analytic structure of functions in Bergman spaces and the behavior of the natural operators acting on them. From the function-theoretic point of view, the reproducing kernel represents point evaluation, and consequently, precise estimates for the kernel translate directly into pointwise bounds, growth properties, boundary behavior, and regularity of Bergman space functions. Off-diagonal estimates, in particular, quantify how values of analytic functions at different points interact, which is essential for understanding localization phenomena and the geometry induced by the space. From the operator-theoretic perspective, many central operators on Bergman spaces -- such as Bergman projections, Toeplitz operators, Hankel operators, and Berezin transforms -- can be written explicitly in terms of the reproducing kernel. Mapping properties, boundedness, compactness, and Schatten-class behavior of these operators depend sensitively on kernel estimates. Moreover, the asymptotic and off-diagonal behavior of Bergman kernels reflects the underlying geometry of the domain and the weight, linking analytic function theory with harmonic analysis. As a result, understanding reproducing kernels provides a unifying framework in which problems from function theory, operator theory, and harmonic analysis on Bergman spaces can be studied simultaneously and effectively.

The purpose of this work is to establish sharp off-diagonal upper estimates for the Bergman reproducing kernels induced by radial doubling weights in the unit disc, and to demonstrate some of their potent applications to operator theory. We begin with necessary definitions. Let~$\H(\D)$ denote the space of all analytic functions in the unit disc $\D=\{z:|z|<1\}$. For~$0<p<\infty$ and a non-negative integrable function $\om$ on $\D$, which we call a weight, the Lebesgue space $L^p_\om$ consists of complex-valued measurable functions $f$ such that
	$$
  \|f\|_{L^p_\omega}^p=\int_\D|f(z)|^p\omega(z)\,dA(z)<\infty,
  $$
where $dA$ denotes the normalized Lebesgue area measure on $\D$. The Bergman space~$A^p_\omega$ is defined as $L^p_\om\cap \H(\D)$.

For each radial weight $\om$, that is $\omega(z)=\omega(|z|)$ for all $z\in\D$, the norm convergence in $A^p_\om$ implies the uniform convergence on compact subsets of~$\D$, and hence each point evaluation $L_z:f\mapsto f(z)$ is a bounded linear functional on $A^p_\om$. The Hilbert space case $p=2$ guarantees the existence of the Bergman reproducing kernels $B^\om_z\in A^2_\om$ satisfying
	\begin{equation}\label{Eq:reproducing}
	L_z(f)
	=f(z)
	=\langle f, B^\om_z\rangle_{A^2_\om}
	=\int_\D f\overline{B^\om_z}\om\,dA,\quad z\in\D,\quad f\in A^2_\om,
	\end{equation}
an identity that remains valid for functions in $A^1_\om$. The kernel has the representation $B^\om_z(\z)=\sum \overline{e_n(z)}e_n(\z)$ for each orthonormal basis $\{e_n\}$ of~$A^2_\om$. This identity yields the formula
	\begin{equation}\label{kernel}
	B^\om_z(\z)
	=\frac12\sum_{n=0}^\infty\frac{\left(\overline{z}\z\right)^n}{\om_{2n+1}},\quad \om_x=\int_0^1r^x\om(r)\,dr,\quad z,\zeta\in\D,
	\end{equation}
when applied to the basis formed by the normalized monomials. In the case of the classical weighted Bergman space $A^2_\alpha$, induced by the standard radial weight $\om_\alpha(z)=(\alpha+1)(1-|z|^2)^\alpha$ with $-1<\alpha<\infty$, this infinite sum reduces to the neat identity
	$$
	B^{\om_\alpha}_z(\zeta)=B^\alpha_z(\zeta)=\frac1{(1-\overline{z}\zeta)^{\alpha+2}},\quad z,\zeta\in\D.
	$$

The case of the diagonal values $z=\zeta$ is straightforward because $B^\om_z(z)$ equals $\|B^\om_z\|^2_{A^2_\om}$. Consequently, the reproducing formula~\eqref{Eq:reproducing} together with the Cauchy-Schwarz inequality yields
	$$
	|B^\om_z(\z)|^2
	\le\|B^\om_z\|_{A^2_\om}^2\|B^\om_\z\|_{A^2_\om}^2
	=B^\om_z(z)B^\om_\zeta(\zeta),\quad z,\z\in\D.
	$$
This trivial estimate provides very limited information about the kernel's off-diagonal behavior, as it gives no control over the dependence on the argument of $\overline{z}\z$. Therefore the real difficulty in obtaining sharp upper estimates for $|B^\om_z(\z)|$ lies in controlling the infinite sum appearing in \eqref{kernel} accurately. To understand this, we consider the standard kernels $B^\alpha_z$ as a toy model, keeping in mind that their neat expression comes from an infinite sum. In view of \eqref{kernel}, each radial weight $\omega$ induces the associated reproducing kernel via its odd moments, and therefore it is natural to expect that the moments somehow concretely dictate the behavior of $|B^\om_z(\z)|$ as well. The moments of the standard radial weights $\omega_\alpha$ satisfy $(\omega_\alpha)_x\asymp x^{-(\alpha+1)}$, and hence we may write
	\begin{equation}\label{adnffdnlffn}
	|B^\alpha_z(\z)|\asymp\frac{1}{|1-\overline{z}\zeta|(\om_\alpha)_{\frac{2}{|1-\overline{z}\zeta|}}},\quad z,\zeta\in\D.
	\end{equation}
This interpretation preserves information on the argument of $\overline{z}\zeta$, relates the pointwise behavior of the kernel to the moments of the inducing weight and causes little loss of information as it is sharp up to a multiplicative constant at every point of $\D\times\D$. One immediately realizes that such an asymptotic relation can not be true at every point in general because kernels might have zeros~\cite{Per}. Motivated by this observation, it is natural to ask under which assumptions on the inducing weight the reproducing kernel satisfies an upper bound of the form given by the right-hand side of \eqref{adnffdnlffn}, and conversely, whether such a pointwise bound imposes structural restrictions on the weight. An ideal situation is one in which the same condition characterizes both directions. In this paper, we show that this is indeed the case: the above pointwise estimate holds precisely for a distinguished class of radial weights, yielding a complete characterization of those weights for which such off-diagonal kernel bounds are valid.

We now proceed to the statement of our main result. We say that a radial weight $\om$ belongs to $\DD$ if there exists a constant $C=C(\om)>0$ such that
	\begin{equation}\label{D-hat}
	\om_x\le C\om_{2x},\quad 1\le x<\infty.
	\end{equation}
It is known that this (one-sided) doubling condition arises naturally in the operator theory of the weighted Bergman spaces \cite{PelRat,PelRatMathAnn,PR2016,PR2021}, and our discoveries will further enforce its role. Namely, this doubling property provides precisely the structural constrain needed to control the infinite sum appearing in~\eqref{kernel} accurately, so as to obtain the right-hand side of~\eqref{adnffdnlffn} as an upper bound. The following theorem is the main result of this paper. Here, as usual, $H^\infty$ denotes the space of bounded analytic functions in $\D$, and $\B$ stands for the classical Bloch space.

\begin{theorem}\label{Theorem:off-diagonal kernel}
Let $\om$ be a radial weight and $n\in\N_0=\N\cup\{0\}$. Then the following statements are equivalent:
	\begin{enumerate}
	\item[\rm(i)] $\displaystyle \left|\left(B^\om_{a}\right)^{(n)}(z)\right|
	\lesssim|a|^n\frac{\kappa(a,z)^{n+1}}{\om_{\kappa(a,z)}},\quad \kappa(a,z)=\frac2{|1-\overline{a}z|},\quad a,z\in\mathbb{D}$;
	\item[\rm(ii)] $\displaystyle \left\|\left(B^\om_{a}\right)^{(n)}\right\|_{H^{\infty}}
	\lesssim \frac{|a|^n}{\om_{\frac{1}{1-|a|}}(1-|a|)^{n+1}},\quad a\in\mathbb{D}$;
	\item[\rm(iii)] $\displaystyle \left\|\left(B^\om_{a}\right)^{(n)}\right\|_{\B}
	\lesssim\frac{|a|^n}{\om_{\frac{1}{1-|a|}}(1-|a|)^{n+1}},\quad a\in\mathbb{D}$;
	\item[\rm(iv)] $\displaystyle\left\|\left(B^\om_{a}\right)^{(n)}\right\|_{A^2_\om}^2
	\lesssim \frac{|a|^{2n}}{\om_{\frac{1}{1-|a|}}(1-|a|)^{2n+1}},\quad a\in\mathbb{D}$;
	\item[\rm(v)] $\om\in\DD$.
	\end{enumerate}
\end{theorem}

It is clear from the proof of the theorem that we may replace ``$\lesssim$'' in the statements~(ii)-(iv) by~``$\asymp$''; however, we have stated the theorem as it is in order to emphasize where the essential information lies. Another immediate remark regarding the statement concerns the fact that the class~$\DD$ can be equally well characterized by the tail integrals $\widehat{\nu}(z)=\int_{|z|}^1\nu(t)\,dt$ by means of the doubling condition $\widehat{\dot\om}(r)\lesssim\widehat{\dot\om}\left(\frac{1+r}{2}\right)$ where $\dot\om(z)=\om(z)|z|$ for all $z\in\D$, and in particular,~$\om\in\DD$ if and only if $\om_{\frac{1}{1-r}}\asymp\widehat{\dot\om}(r)$ by~\cite[Lemma~2.1]{PelSum14}. Therefore, under the hypothesis $\om\in\DD$, we may replace, without any essential loss of information, the moments appearing on the right-hand side of (i)--(iv) by the corresponding tail integrals, that are often simpler to work with. This simple observation has many useful consequences, some of which will be discussed a bit later in this work.

The self-contained proof of Theorem~\ref{Theorem:off-diagonal kernel} is presented in the next section, divided into several parts. Among these, the most delicate step is showing that the pointwise upper estimate (i) holds for each $\omega\in\DD$. Quite surprisingly, this step can be settled by a combination of rather elementary tools such as Lagrange's theorem for backward differences, Fa\`a di Bruno's formula (an identity that generalizes the chain rule to higher derivatives), and known characterizations of weights in $\DD$. Using these basic ingredients, we first prove that (v) implies~(i) in the base case $n=0$. An analogue for higher-order derivatives is then obtained via the identity
	\[
	(B_{z}^{\om})^{(n)}(\z)=\overline{z}^{n}B_{z}^{V_{\om,n}}(\z),\quad z,\z\in\D,\quad n\in\N\cup\{0\},
	\]
which neatly relates the $n$-th derivative of the kernel to another kernel induced by $V_{\om,n}\in\DD$, a weight appropriately associated with $\omega$ and $n$. Then trivially (i) implies (ii), and the Schwarz-Pick theorem shows that (iii) follows from (ii), while (v) is obtained from (iii) through direct but elementary computations. Finally, the equivalence between (iv) and (v) is a consequence of Parseval's identity and pretty straightforward estimates. Several known characterizations of weights in~$\DD$ are also used throughout the proof. In Section~\ref{sec4} we demonstrate that the method employed in the proof of Theorem~\ref{Theorem:off-diagonal kernel} is immune to the underlying dimension, works for harmonic Bergman kernels, and applies equally well for the two-weight fractional derivatives instead of the standard ones, yielding sharp upper estimates for Dirichlet reproducing kernels.

H.~Hedenmalm~\cite{H} established an off-diagonal estimate for the Bergman reproducing kernels induced by weights $\omega$ for which $\log\omega$ is subharmonic, refining an earlier result of S.~Shimorin~\cite{S}, by showing that
	\begin{equation}\label{HH}
	\frac12B^\om_a(a)\frac{(1-|a|^2)^2}{|1-\overline{a}z|^2}\le|B^\om_a(z)|\le\frac{2}{|1-\overline{a}z|^2},\quad a,z\in\D.
	\end{equation}
Another estimate of similar nature was obtained in \cite[Proposition~4.5]{S2}. The differences between \eqref{HH} and Theorem~\ref{Theorem:off-diagonal kernel} are fundamental. First, the upper (as well as the lower) bound in \eqref{HH} depends only on the class of weights, with a precise constant factor, whereas the estimate in Theorem~\ref{Theorem:off-diagonal kernel}(i) is accurately dictated by the weight itself. Moreover, the kernels induced by logarithmically subharmonic weights are zero-free~\cite{DKS,H,S2}, as the lower bound in \eqref{HH} also confirms, while the kernels associated with weights in $\DD$ are certainly not in general~\cite{Per}. Therefore no global lower bound for $|B^\om_z(\z)|$ can be expected for arbitrary pairs $(z,\z)\in\D\times\D$ when $\omega\in\DD$, but sharp local lower estimates for $|B^\om_z(\z)|$ near the diagonal can be found in \cite[Lemma~7 and Lemma~8]{PRS2}. Second, the subharmonicity of $\log \omega$ is definitely a local property while weights in $\DD$ do not exhibit any such regularity because they are defined by the doubling condition~\eqref{D-hat} on moments, and, in fact, they may admit a strong oscillatory behavior and may even vanish on a large part of each outer annulus of $\D$ by \cite[Proposition~3]{PR2023}. It is still worth emphasizing that, unlike weights in $\DD$, logarithmically subharmonic weights are not radial and that is another significant difference between these two classes. Therefore it does not come as a surprise that the deep methods developed in \cite{H}, \cite{S} and \cite{S2} are more involved than those used to establish Theorem~\ref{Theorem:off-diagonal kernel}, although obviously they do not apply in the setting of $\DD$ due to the absence of local control over the weights.

We now turn to demonstrate the usefulness of our findings by presenting several of their consequences. Our first goal is to apply Theorem~\ref{Theorem:off-diagonal kernel} to establish a fairly sharp estimate for the $L^p$-means of certain modified Bergman kernels. For $n\in\N_0$ and $\beta\in\RR$, denote
		$$
		(B^{\om}_a)_{[\beta]}^{(n)}(z)
		=(B^{\om}_{a})^{(n)}(z)(1-\overline{a}z)^{\beta},\quad a,z\in\D,
		$$
and, for $0<p<\infty$ and $0<r<1$, define
		$$
		M_p^p(r,f)=\frac{1}{2\pi}\int_{0}^{2\pi}|f(re^{it})|^p\,dt,\quad f\in\H(\D).
		$$
Recall that $\om\in\DD$ if and only if $\om_{\frac{1}{1-r}}\asymp\widehat{\dot\om}(r)$ for all $0\le r<1$~\cite[Lemma~2.1]{PelSum14}.
This explains the appearance of $\widehat{\dot\om}(t)$ in the statement of Corollary~\ref{Corollary:Integrated kernel estimates}, and shows that it can be replaced by the moment $\om_{\frac1{1-t}}$.

\begin{corollary}\label{Corollary:Integrated kernel estimates}
Let $0<p<\infty$, $n\in\N\cup\{0\}$, $\beta\in\RR$ and $\om\in\DD$, and let $\eta$ be a radial weight. Then
	$$
	M_p^p\left(r,\left(B^{\om}_a\right)_{[\beta]}^{(n)}\right)
	\lesssim|a|^{np}\left(\int_0^{|a|r}\frac{dt}{\widehat{\dot\om}(t)^p(1-t)^{p(n+1-\beta)}}+1\right),\quad a\in\D,
	$$
for all $0<r<1$, and
	$$
	\left\|\left(B^{\om}_a\right)_{[\beta]}^{(n)}\right\|_{A^p_\eta}^p
	\lesssim|a|^{np}\left(\int_0^{|a|}\frac{\widehat{\dot\eta}(t)}{\widehat{\dot\om}(t)^p(1-t)^{p(n+1-\beta)}}\,dt+1\right),\quad a\in\D.
	$$
\end{corollary}

These kind of sharp integrated kernel estimates can be interpreted as generalized Forelli-Rudin estimates and have proven highly useful in the operator theory of Bergman spaces induced by doubling weights. They have played a central role in studies of the Bergman projection \cite{PR2016,PR2021} and the closely related Hankel and Toeplitz operators~\cite{DGRW,PR2016/2,PR2025}, as well as in many other instances \cite{DLLS,LW2025,MP2026,MPR2024,PRW}. One can fairly say that, in most cases, any operator involving the modulus of the kernel or its derivative can be simplified without any essential loss of information by using the estimates given in Theorem~\ref{Theorem:off-diagonal kernel}, and $L^p$-integrals of the kernels themselves are efficiently treated with Corollary~\ref{Corollary:Integrated kernel estimates}.

The special case $\beta=0$ of Corollary~\ref{Corollary:Integrated kernel estimates} was established a decade ago in \cite[Theorem~1]{PR2016}, while the recent result \cite[Lemma~5]{PR2025} gives the latter statement for $n=0$, $\beta\in\N$ and $2\le p<\infty$. In contrast to the relatively involved arguments required in \cite{PR2016,PR2025}, the proof of Corollary~\ref{Corollary:Integrated kernel estimates} becomes remarkably straightforward once Theorem~\ref{Theorem:off-diagonal kernel}(i) is established.

Our next application of Theorem~\ref{Theorem:off-diagonal kernel} establishes a new connection between the H\"ormander maximal function and the Berezin transform of an $L^1_\om$-function. This relation may be regarded as an analogue of the classical interplay between the Poisson integral and the Hardy-Littlewood maximal function of an integrable function on the unit circle $\T$. More precisely, for $\varphi\in L^1(\T)$, a fundamental inequality in this setting states that the non-tangential maximal function of the Poisson integral $P(\varphi)$ is dominated by a universal constant times the Hardy-Littlewood maximal function $\M(\varphi)$ on $\T$~\cite[Propositions 7.11.1 and 7.1.7]{Pabook1}. In the setting of $L^1_\om$ over the disc, the H\"ormander maximal function often plays a role similar to that of the Hardy-Littlewood maximal function on the boundary. Recall that the H\"ormander maximal function~\cite{HormanderL67} of $f\in L^1_\om$ is defined by
	$$
	M_{\om}(f)(z)=\sup_{z\in S(a)}\frac{1}{\om\left(S(a)\right)}\int_{S(a)}|f(\z)|\om(\z)\,dA(\z),\quad z\in\D,
	$$
where $S(a)=\left\{re^{it}\in\D: 1-|a|\le r<1, ~|t-\arg a|\le \frac{1-|a|}{2}\right\}$ denotes the Carleson square for $a\in\D\setminus\{0\}$ and $S(0)=\D$. Further, the Berezin transform
		$$
		B_\om(f)(z)=\int_\D f(\zeta)\frac{|B^\om_z(\zeta)|^2}{\|B^\om_z\|_{A^2_\om}^2}\om(\zeta)\,dA(\zeta),\quad z\in\D,
		$$
serves as a natural analogue of the Poisson integral, even though it is not necessarily harmonic. For $0<M<\infty$ and a measurable function $f$, the non-tangential maximal function of $f$ in the (punctured) unit disc is defined as
		$$
		N_M(f)(\z)=\sup_{z\in\G_M(\z)}|f(z)|,\quad \z\in\D\setminus\{0\},
		$$
where $\G_M(\z)=\left\{z\in\D: |\arg\z-\arg z|<M\left(|\z|-{|z|}\right)\right\}$ is a Bergman cone with vertex at $\zeta$.
Our next result shows that the H\"ormander maximal function of $f\in L^1_\om$ dominates the non-tangential maximal function of the Berezin transform $B_\om(|f|)$, and therefore also dominates the Berezin transform itself. To see the broader picture, we consider the operator
	$$
	T_{\delta,\om}(f)(z)
	=\int_\D|f(\zeta)|\frac{|B^\om_z(\zeta)|^{1+\delta}}{\|B^\om_z\|_{A^2_\om}^{2\delta}}\om(\zeta)\,dA(\zeta),\quad z\in\D,\quad f\in L^1_\om,
	$$
where $0\le \delta<\infty$. Clearly, $T_{1,\om}(f)=B_\om(|f|)$ and $T_{0,\om}(f)$ reduces to the maximal Bergman projection $P^+_\omega(|f|)$.

\begin{corollary}\label{Corollary:Maximal function of Berezin transform}
Let $0<\delta,M<\infty$ and $\om\in\DD$. Then $N_M(T_{\delta,\om}(f))\lesssim M_\om(f)$ on $\D$ for all $f\in L^1_\om$. In particular, the non-tangential maximal function of the Berezin transform is dominated by the H\"ormander maximal function, that is, $N_M(B_\om(|f|))\lesssim M_\om(f)$ on~$\D$ for all $f\in L^1_\om$.
\end{corollary}

If the statement of Corollary~\ref{Corollary:Maximal function of Berezin transform} were true for $\delta=0$, it would imply that the maximal Bergman projection $P^+_\om:L^p_\om\to L^p_\om$ is bounded for each $\om\in\DD$ because $M_\omega$ is bounded on $L^p_\omega$ by~\cite{PelRatMathAnn, PRS1}. However, this is known to be false because for $1<p<\infty$ and $\om\in\DD$, $P^+_\om:L^p_\om\to L^p_\om$ is bounded if and only if $\om\in\DD$ satisfies an additional reverse doubling property~\cite[Theorem~9]{PR2021}. Therefore the statement of Corollary~\ref{Corollary:Maximal function of Berezin transform} fails for $\delta=0$. The proof of the corollary is based on a direct application of the pointwise estimate of Theorem~\ref{Theorem:off-diagonal kernel} and standard geometric arguments involving a suitable partition of the disc.

Corollary~\ref{Corollary:Maximal function of Berezin transform} allows us to characterize the boundedness of $\left(T_{\delta,\om}\left(|\cdot|^{\frac{1}{\alpha}}\right)\right)^{\alpha}:L^p_\omega\to L^p_\mu$ when $\delta>0$ in terms of Carleson measures, provided $\alpha p>1$. This follows because $\left(M_{\om}\left(|\cdot|^{\frac{1}{\alpha}}\right)\right)^{\alpha}:L^p_\omega\to L^p_\mu$ is bounded if and only if $\sup_{S}\frac{\mu(S)}{\om(S)}<\infty$, provided $0<p<\infty$, $\alpha p>1$ and $\om\in\DD$, which is also equivalent to saying that $
\mu$ is a $p$-Carleson measure for $A^p_\om$, i.e. the identity mapping $I:A^p_\om\to L^p_\mu $ is bounded~\cite{PelRatMathAnn, PRS1}. In particular, $M_{\om}$ is bounded on~$L^p_\om$ for each fixed~$1<p<\infty$ and~$\om\in\DD$, and so must be $T_{\delta,\om}$ by Corollary~\ref{Corollary:Maximal function of Berezin transform}.

\begin{corollary}\label{Corollary:B-boundedness}
Let $0<p,\alpha<\infty$, $0<\delta<\infty$ and $\om\in\DD$ such that $\alpha p>1$, and let $\mu$ be a positive Borel measure on $\D$. Then the following statements are equivalent:
	\begin{itemize}
	\item[\rm(i)] $\left(M_{\om}\left(|\cdot|^{\frac{1}{\alpha}}\right)\right)^{\alpha}:L^p_\omega\to L^p_\mu$ is bounded;
	\item[\rm(ii)] $\left(T_{\delta,\om}\left(|\cdot|^{\frac{1}{\alpha}}\right)\right)^{\alpha}:L^p_\omega\to L^p_\mu$ is bounded;
	\item[\rm(iii)] $I:A^p_\om\to L^p_\mu$ is bounded;
	\item[\rm(iv)] $\mathcal{S}=\displaystyle\sup_{S}\frac{\mu(S)}{\om(S)}<\infty$.
	\end{itemize}
Moreover,
    \begin{equation}\label{Eq:operator norm estimates*}
		\begin{split}
    \left\|\left(M_{\om}\left(|\cdot|^{\frac{1}{\alpha}}\right)\right)^{\alpha}\right\|^p_{L^p_\om\to L^p_\mu}
		&\asymp\left\|\left(T_{\delta,\om}\left(|\cdot|^{\frac{1}{\alpha}}\right)\right)^{\alpha}\right\|^p_{L^p_\om\to L^p_\mu}\\
		&\asymp\|I\|_{A^p_\om\to L^p_\mu}^p\asymp\mathcal{S}.
		\end{split}
		\end{equation}
\end{corollary}

The connection between (ii) and (iii) in the statement of Corollary~\ref{Corollary:B-boundedness} becomes apparent in the case $p>\alpha=1=\delta$. Namely, the reproducing formula for functions in $A^1_\om$ yields $B_\om(f)=f$, and therefore $I:A^p_\om\to L^p_\mu$ is bounded whenever $B_\om:L^p_\om\to L^p_\mu$ is bounded, and $\|I\|_{A^p_\om\to L^p_\mu}\le\|B_\om\|_{L^p_\om\to L^p_\mu}$. The special case $\delta=1$ of Corollary~\ref{Corollary:B-boundedness} improves the recent results \cite[Theorem~5.4]{HJLL} and \cite[Theorem~1]{HLPRW2026} in part by substantially relaxing the hypothesis on the weights considered.

The remainder of the paper is organized as follows. Section~\ref{sec2} is devoted to the proof of our main result, Theorem~\ref{Theorem:off-diagonal kernel}, while Section~\ref{sec3} contains the proofs of Corollaries~\ref{Corollary:Integrated kernel estimates}--\ref{Corollary:B-boundedness}. In Section~\ref{sec4}, we present further remarks on the method developed on the way to the proof of the main result, extending it to two-weight fractional derivatives of kernels, to higher dimensions, and to harmonic Bergman kernels.

We finish the introduction with a couple of words about the notation used throughout the paper. The letter $C=C(\cdot)$ denotes an absolute constant whose value depends on the parameters indicated in the parenthesis and may change from one occurrence to another. If there exists a constant $C=C(\cdot)>0$ such that $a\le Cb$, then we write either $a\lesssim b$ or $b\gtrsim a$. In particular, if $a\lesssim b$ and $a\gtrsim b$, then we denote $a\asymp b$ and say that $a$ and $b$ are comparable.

\section{Proof of Theorem~\ref{Theorem:off-diagonal kernel}}\label{sec2}

To prove our main result, we will need a special case of the mean value theorem for divided differences~\cite[p.~61]{B2005}, commonly used in error analysis in interpolation and approximation theory. When the data points are equispaced, it reduces to so-called Lagrange's theorem, which states that, if $N\in\N$, $h\in\RR$ and a complex-valued function $f\in C^N(\RR)$ are given, then there exists $c=c(f,x,N,h)$ on the closed interval with endpoints $x-Nh$ and~$x$ such that
	$$
	\nabla_h^N(f)(x)=h^Nf^{(N)}(c).
	$$
Here $\nabla_h(f)(x)=f(x)-f(x-h)$ is the backward difference of a real-variable function~$f$ at~$x$, and
		\begin{equation*}
    \nabla_h^N(f)(x)
		=\left(\nabla_h\circ\cdots\circ\nabla_h\right)(f)(x)
		=\sum_{m=0}^N(-1)^m\binom{N}{m}f(x-mh),\quad x,h\in\RR,
		\end{equation*}
as a direct calculation shows.

\medskip

\noindent\emph{Proof of Theorem~\ref{Theorem:off-diagonal kernel}}. Let first $\om\in\DD$, and consider the case $n=0$ in (i). Define $\nu(r)=\om(\sqrt{r})$ for all $0\le r<1$. Then
	\begin{equation}\label{om-nu}
	2\om_{2x+1}
	=\nu_x, \quad 0\le x<\infty,
	\end{equation}
and hence
	$$
	B^\om_a(z)
	=\sum_{n=0}^\infty\frac{(\overline{a}z)^n}{2\om_{2n+1}}
	=\sum_{n=0}^\infty\frac{(\overline{a}z)^n}{\nu_{n}},\quad a,z\in\D.
	$$
Write $\overline{a}z=\zeta$, and denote $\kappa(\zeta)=\frac{2}{|1-\zeta|}$ and
	$$
	f_\nu(\zeta)=\sum_{n=0}^\infty\frac{\zeta^n}{\nu_{n}},\quad \zeta\in\D.
	$$
Further, write $E=E(\zeta)=\lfloor\kappa(\zeta)\rfloor\in\N$, where $\lfloor\kappa\rfloor$ is the largest integer less than or equal to $\kappa\in\RR$. Then
    \begin{equation*}
    f_\nu(\zeta)
    =\sum_{n=0}^{E-1}\frac{\zeta^n}{\nu_n}
    +\sum_{n=E}^{\infty}\frac{\zeta^n}{\nu_n}
		=S_1(E,\zeta)+S_2(E,\zeta),
		\end{equation*}
where
		\begin{equation*}
    |S_1(E,\zeta)|
		\le\sum_{n=0}^{E-1}\frac{1}{\nu_n}
		\le\frac{E}{\nu_{E-1}}
		\le\frac{\kappa(\zeta)}{\nu_{\kappa(\zeta)}},\quad \zeta\in\D.
		\end{equation*}
Let $N\in\N$ to be fixed later. Then the identity $(1-\zeta)^N=\sum_{k=0}^{N}(-1)^k\binom{N}{k}\zeta^k$ and Fubini's theorem yield
		\begin{equation*}
		\begin{split}\allowdisplaybreaks
    |S_2(E,\zeta)(1-\zeta)^N|
    &=\left|\sum_{k=0}^{N}(-1)^k\binom{N}{k}\left(\sum_{n=E+k}^{E+N}+\sum_{n=E+N+1}^{\infty}\right)\frac{\zeta^{n}}{\nu_{n-k}}\right|\\
    &\le\left|\sum_{n=E}^{E+N}\left(\sum_{k=0}^{n-E}(-1)^k\binom{N}{k}\frac{1}{\nu_{n-k}}\right)\zeta^{n}\right|\\
		&\quad+\left|\sum_{n=E+N+1}^{\infty}\left(\sum_{k=0}^{N}(-1)^k\binom{N}{k}\frac{1}{\nu_{n-k}}\right)\zeta^{n}\right|\\
		&=|S_{1,N}(E,\zeta)|+|S_{2,N}(E,\zeta)|,\quad \zeta\in\D.
		\end{split}
		\end{equation*}
Denote $W(x)=W_\nu(x)=\frac{1}{\nu_x}$ for all $x\ge0$, and observe that the inner sum in the last expression is the $N$-th backward difference $\nabla^N(W)$ of the function $W$ at the point~$n$. Therefore Lagrange's theorem ensures the existence of $h_n^N \in [n-N,n]$ such that
	$$
	\nabla^N(W)(n)
	=\sum_{k=0}^{N}(-1)^k\binom{N}{k}\frac{1}{\nu_{n-k}}
	=W^{(N)}(h_n^N), \quad n\ge N.
	$$
Assume for a moment that
		\begin{equation}\label{derivative-estimate}
    |W^{(N)}(x)|\lesssim\frac{1}{\nu_xx^N}, \quad 1\le x<\infty,
		\end{equation}
for each fixed $N\in\N$. Since $\omega$ is doubling by the assumption, so is $\nu$ by \eqref{om-nu}. Therefore \cite[Lemma~2.1(ii)(vi)]{PelSum14} ensures the existence of a $\beta=\beta(\nu)>0$ such that $x\mapsto x^\beta\nu_x$ is essentially increasing on $[1,\infty)$. Define $N=\lfloor\beta\rfloor+2$. Then
    \begin{equation*}
		\begin{split}
    |S_{2,N}(E,\zeta)|
    &\le\sum_{n=E+N+1}^{\infty}|W^{(N)}(h_n^N)||\zeta|^n
    \lesssim\sum_{n=E+N+1}^{\infty}\frac{1}{\nu_nn^N}\\
		&\lesssim\frac{1}{\nu_{E}E^\beta}\sum_{n=E+N+1}^{\infty}\frac{1}{n^{N-\beta}}
    \asymp\frac{\kappa(\zeta)^{1-N}}{\nu_{\kappa(\zeta)}},\quad \zeta\in\D.
		\end{split}
		\end{equation*}
To establish the same upper bound for $|S_{1,N}(E,\zeta)|$, we proceed by induction. Let first $N=2$. Then
		\begin{equation*}
    \begin{split}
		S_{1,2}(E,\zeta)
		&=\frac{\zeta^{E}}{\nu_{E}}
		+\left(\frac{1}{\nu_{E+1}}-\frac{2}{\nu_{E}}\right)\zeta^{E+1}
		+\sum_{k=0}^2(-1)^k\binom{2}{k}\frac{1}{\nu_{E+2-k}}\zeta^{E+2},\quad \zeta\in\D,
		\end{split}
		\end{equation*}
and hence
		\begin{equation}\label{sdgf}
		\begin{split}
    |S_{1,2}(E,\zeta)|
		&\le\frac{|\zeta|^E}{\nu_{E}}\left|1-\left(2-\frac{\nu_{E}}{\nu_{E+1}}\right)\zeta\right|
		+\left|\frac{1}{\nu_{E+2}}-\frac{2}{\nu_{E+1}}+\frac{1}{\nu_{E}}\right|\\
    &\le\frac{1}{\nu_{E}}\left|1-\left(2-\frac{\nu_{E}}{\nu_{E+1}}\right)\zeta\right|
		+\left|\frac{1}{\nu_{E+2}}-\frac{1}{\nu_{E+1}}\right|\\
		&\quad+\left|\frac{1}{\nu_{E+1}}-\frac{1}{\nu_{E}}\right|,\quad \zeta\in\D.
    \end{split}
		\end{equation}
These three terms are now estimated separately. Since $\nu\in\DD$ by the assumption, for each $\alpha>0$ there exists a constant $C=C(\nu,\alpha)>0$ such that
	\begin{equation}\label{ksksksk*}
	x^{\alpha}(\nu_{[\alpha]})_x\le C \nu_x, \quad 0< x < \infty,
	\end{equation}
where $\nu_{[\alpha]}$ is defined by $\nu_{[\alpha]}(r)=\nu(r)(1-r)^\alpha$. For this fact and more, see \cite[(1.3)]{PR2021}. For the first term on the right-hand side of \eqref{sdgf} we observe that \eqref{ksksksk*} with $\alpha=1$ yields
	\begin{equation*}
	\begin{split}
	\left|1-\left(2-\frac{\nu_{E}}{\nu_{E+1}}\right)\zeta\right|
	&\le|1-\zeta|+|\zeta|\left(\frac{\nu_{E}}{\nu_{E+1}}-1\right)\\
	&\le|1-\zeta|+\frac{\left(\nu_{[1]}\right)_{E}}{\nu_{E+1}}
	\lesssim\frac1{\kappa(\zeta)},\quad \zeta\in\D.
	\end{split}
	\end{equation*}
This together with equally natural applications of \eqref{ksksksk*} to the other two terms in \eqref{sdgf} yield
	$$
	|S_{1,2}(E,\zeta)|\lesssim\frac1{\nu_{\kappa(\zeta)}\kappa(\zeta)},\quad \zeta\in\D,
	$$
and thus the case $N=2$ is proved. Assume next that
	\begin{equation}\label{Induction---}
	|S_{1,N}(E,\zeta)|
	\lesssim\frac{\kappa(\zeta)^{1-N}}{\nu_{\kappa(\zeta)}},\quad \zeta\in\D,
	\end{equation}
for a fixed $N\in\N$. A direct calculation yields
		\begin{equation*}
    \begin{split}\allowdisplaybreaks
    S_{1,N+1}(E,\zeta)
		&=\frac{\zeta^{E}}{\nu_{E}}
		+\sum_{n=E+1}^{E+N}\left(\sum_{k=0}^{n-E}(-1)^k\binom{N+1}{k}\frac{1}{\nu_{n-k}}\right)\zeta^n\\
		&\quad+\sum_{k=0}^{N+1}(-1)^k\binom{N+1}{k}\frac{\zeta^{E+N+1}}{\nu_{E+N+1-k}}\\
		&=\frac{\zeta^{E}}{\nu_{E}}+\sum_{n=E+1}^{E+N}\frac{\zeta^n}{\nu_{n}}
		+\sum_{n=E+1}^{E+N}\left(\sum_{k=1}^{n-E}(-1)^k\binom{N+1}{k}\frac{1}{\nu_{n-k}}\right)\zeta^n\\
		&\quad+\sum_{k=0}^{N+1}(-1)^k\binom{N+1}{k}\frac{\zeta^{E+N+1}}{\nu_{E+N+1-k}},\quad \zeta\in\D.
    \end{split}
		\end{equation*}
Then the identity $\binom{N+1}{k}=\binom{N}{k}+\binom{N}{k-1}$ together with standard computations imply
	\begin{allowdisplaybreaks}
		\begin{eqnarray*}
    &&S_{1,N+1}(E,\zeta)=(1-\zeta)S_{1,N}(E,\zeta)\\
		&&\quad+\zeta^{E+N+1}\left(\sum_{k=0}^{N+1}(-1)^k\binom{N+1}{k}\frac{1}{\nu_{E+N+1-k}}
		+\sum_{k=0}^N(-1)^k\binom{N}{k}\frac{1}{\nu_{E+N-k}}\right)\\
		&&\quad=(1-\zeta)S_{1,N}(E,\zeta)+\zeta^{E+N+1}S_{3,N}(E),\quad \zeta\in\D.
		\end{eqnarray*}
		\end{allowdisplaybreaks}
The first term is exactly of the type we want. To deal with the second term, we observe that
		\begin{equation*}
    \begin{split}
		S_{3,N}(E)
		&=\sum_{k=0}^{N}(-1)^k\binom{N}{k}\frac{1}{\nu_{E+N+1-k}}.
    \end{split}
		\end{equation*}
By Lagrange's theorem there exists an $h_{E,N}^N\in[E+1,E+N+1]$ such that
	$$
	\sum_{k=0}^{N}(-1)^k\binom{N}{k}\frac{1}{\nu_{E+N+1-k}}=W^{(N)}(h_{E,N}^N).
	$$
By combining these identities we deduce
	\begin{equation*}
  \begin{split}
  S_{1,N+1}(E,\zeta)
	&=(1-\zeta)S_{1,N}(E,\zeta)+\zeta^{E+N+1}W^{(N)}(h_{E,N}^N),\quad\zeta\in\D,
	\end{split}
	\end{equation*}
and hence the induction hypothesis and \eqref{derivative-estimate} finally give
	\begin{equation*}
  \begin{split}
  |S_{1,N+1}(E,\zeta)|
	&\le|1-\zeta||S_{1,N}(E,\zeta)|+|W^{(N)}(h_{E,N}^N)|\\
	&\lesssim|1-\zeta|\frac{\kappa(\zeta)^{1-N}}{\nu_{\kappa(\zeta)}}
	+\frac{1}{\nu_{E}E^N}
	\asymp\frac{\kappa(\zeta)^{-N}}{\nu_{\kappa(\zeta)}},\quad\zeta\in\D.
	\end{split}
	\end{equation*}
Thus \eqref{Induction---} is valid for every fixed $N\in\N$ by induction.
		
It remains to prove \eqref{derivative-estimate}. Let $I_{\lambda}(N)$ denote the set of $\lambda=(\lambda_1,\lambda_2,\ldots,\lambda_N)\in\N_0^{N}$ such that $\sum_{j=1}^{N}j\lambda_j=N$, and define $|\lambda|=\sum_{j=1}^{N}\lambda_j$. Fa\`a Di Bruno's formula states that
		\begin{equation*}
		\begin{split}
    W^{(N)}(x)
		=\sum_{I_{\lambda}(N)}\frac{N!}{\prod_{j=1}^{N}\lambda_{j}!j!^{\lambda_j}}
		\frac{(-1)^{|\lambda|}|\lambda|!}{\nu_{x}^{|\lambda|+1}}\prod_{k=1}^{N}\left(\frac{d^{k}}{dx^{k}}\nu_x\right)^{\lambda_k}.
		\end{split}
		\end{equation*}
By differentiating under the integral sign and using \eqref{ksksksk*}, we obtain
		\begin{equation*}
		\begin{split}
    \left|\frac{d^{k}}{dx^{k}}\nu_x\right|
		&=\int_{0}^{1}r^{x}\left(\log\frac{1}{r}\right)^k\nu(r)\,dr
    \asymp\left(\nu_{[k]}\right)_x
		\lesssim\frac{\nu_{x}}{x^k}, \quad 1\le x<\infty,
		\end{split}
		\end{equation*}
and hence
		\begin{equation*}
		\begin{split}
    |W^{(N)}(x)|
		&\lesssim\sum_{I_{\lambda}(N)}\frac{N!}{\prod_{j=1}^{N}\lambda_{j}!j!^{\lambda_j}}\frac{|\lambda|!}{\nu_{x}^{|\lambda|+1}}
		\prod_{k=1}^{N}\left(\frac{\nu_{x}}{x^{k}}\right)^{\lambda_k}\\
    &=\sum_{I_{\lambda}(N)}\frac{N!}{\prod_{j=1}^{N}\lambda_{j}!j!^{\lambda_j}}\frac{|\lambda|!}{\nu_{x}^{|\lambda|+1}}\frac{\nu_x^{|\lambda|}}{x^N}
		\asymp\frac{1}{\nu_x x^N},\quad 1\le x<\infty,
		\end{split}
		\end{equation*}
for each fixed $N\in\N$. Therefore the estimate \eqref{derivative-estimate} is proved. In view of the identity $\nu_x=2\om_{2x+1}$ this shows that
	$$
	|B^\om_a(z)|
	=|f_\nu(\zeta)|
	\lesssim\frac{\kappa(\zeta)}{\nu_{\kappa(\zeta)}}
	\asymp\frac{\kappa(a,z)}{\om_{\kappa(a,z)}},\quad \kappa(a,z)=\frac2{|1-\overline{a}z|},\quad a,z\in\mathbb{D},
	$$
and thus we have shown that (i) with $n=0$ is satisfied when $\om\in\DD$.

To prove the statement (i) for $n\in\N$, we show that the $n$-th derivative of the Bergman reproducing kernel $B_{z}^{\om}$ is $\overline{z}^{n}$ times another Bergman kernel induced by a weight depending on $\om$ and $n$. To see this, for a radial weight $\om$, define $V_{\om,0}=\om$, and set $V(r)=V_{\om,1}(r)=2\int_{r}^{1}\om(t)t\,dt$ for all $0\le r<1$. Moreover, set
	\begin{equation*}
	\begin{split}
	&V_{\om,n}(z)\\
	&=2^n\int_{|z|}^1r_1\left(\int_{r_1}^1r_2\cdots\left(\int_{r_{n-2}}^1r_{n-1}
	\left(\int_{r_{n-1}}^1\om(r)r\,dr\right)dr_{n-1}\right)\cdots\,dr_2\right)dr_1
	\end{split}
	\end{equation*}
for all $z\in\D$ and $n\in\N$. A direct calculation shows that $(k+1)V_{2k+1}=\om_{2k+3}$ for all $k\in\mathbb{N}\cup\{0\}$, and hence
	\begin{equation*}
	\begin{split}
	(B_{z}^{\om})'(\z)
	=\overline{z}\sum_{k=0}^{\infty}\frac{(k+1)(\overline{z}\z)^{k}}{2\om_{2k+3}}
	=\overline{z}B_{z}^{V}(\z),\quad z,\zeta\in\D.
	\end{split}
	\end{equation*}
By iterating this formula we obtain
	\[
	(B_{z}^{\om})^{(n)}(\z)=\overline{z}^{n}B_{z}^{V_{\om,n}}(\z),\quad z,\z\in\D,\quad n\in\N\cup\{0\}.
	\]
Observe that $\left(V_{\om,n}\right)_x\asymp\frac{\om_x}{x^{n}}$ for $1\le x<\infty$. It then follows from the case $n=0$ that
	$$
	\left|\left(B^\om_{a}\right)^{(n)}(z)\right|\lesssim|a|^n\frac{\kappa(a,z)^{n+1}}{\om_{\kappa(a,z)}},\quad a,z\in\D.
	$$
Therefore we have shown that (v) implies (i).

It is clear that (i) implies
	\begin{equation}\label{(ii')}
	\left\|\left(B^\om_{a}\right)^{(n)}\right\|_{H^{\infty}}
	\lesssim \frac{|a|^n}{\om_{\frac{2}{1-|a|}}(1-|a|)^{n+1}},\quad a\in\mathbb{D},
	\end{equation}
from which we deduce
	\begin{equation}\label{(iii')}
	\left\|\left(B^\om_{a}\right)^{(n)}\right\|_{\B}
	\lesssim\frac{|a|^n}{\om_{\frac{2}{1-|a|}}(1-|a|)^{n+1}},\quad a\in\mathbb{D},
	\end{equation}
by the Schwarz-Pick theorem. Observe that \eqref{(ii')} and \eqref{(iii')} differ from the statements (ii) and (iii) only by the constant factor 2 appearing in the moment index. We will handle this minor issue a bit later.

Assume next \eqref{(iii')} for some $n\in\N\cup\{0\}$. Then
	\begin{equation*}
	\begin{split}
	\frac{|a|^n}{\om_{\frac{2}{1-|a|}}(1-|a|)^{n+1}}
	\gtrsim\left\|\left(B^\om_{a}\right)^{(n)}\right\|_{\B}
	&\ge\frac{|a|^{n+1}(1-|a|)}{2}\sum_{k=1}^\infty\frac{k^{n+1}|a|^{2(k-1)}}{\om_{2(k+n)+1}},\quad a\in\D,
	\end{split}
	\end{equation*}
and hence
	\begin{equation}\label{d.f,ns.,m,mmn}
	\begin{split}
	\frac{1}{\om_{\frac{2}{1-|a|}}(1-|a|)^{n+2}}
	\gtrsim\sum_{k=1}^\infty\frac{k^{n+1}|a|^{2k}}{\om_{2k}},\quad a\in\D.
	\end{split}
	\end{equation}
Let $1<K_1<K_2<\infty$ and observe that
	\begin{equation}\label{lkjflerjlqkerjnv}
	\begin{split}
	\sum_{k=1}^\infty\frac{k^{n+1}|a|^{2k}}{\om_{2k}}
	&\gtrsim\int_1^\infty\frac{x^{n+1}|a|^{2x}}{\om_{2x}}\,dx
	\ge\int_{1-\frac{1-|a|}{K_1}}^{1-\frac{1-|a|}{K_2}}\frac{|a|^{\frac2{1-t}}}{\om_{\frac{2}{1-t}}}\frac{dt}{(1-t)^{n+3}}\\
	&\ge\frac{|a|^{\frac{2K_2}{1-|a|}}}{\om_{\frac{2K_1}{1-|a|}}}\frac{K_1^{n+3}}{(1-|a|)^{n+3}}(1-|a|)\left(\frac1{K_1}-\frac1{K_2}\right)\\
	&\ge\frac{|a|^{\frac{2K_2}{1-|a|}}}{\om_{\frac{2K_1}{1-|a|}}}\frac{1}{(1-|a|)^{n+2}}\frac{K_2-K_1}{K_1K_2},\quad a\in\D.
	\end{split}
	\end{equation}
Choose $K_2=2K_1=4$ and set $x=\frac{2}{1-|a|}$. Then \eqref{d.f,ns.,m,mmn} and \eqref{lkjflerjlqkerjnv} yield $\om_x\lesssim\om_{2x}$ for $x\ge2$, and it follows that $\om\in\DD$. Therefore \eqref{(iii')} implies (v), and thus (i), \eqref{(ii')}, \eqref{(iii')} and (v) are equivalent. From this it easily follows that so are the statements (i)--(iii) and (v) of the theorem.

It remains to show that (iv) is equivalent to the other statements. Parseval's identity for the Hardy space $H^2$ gives
	\begin{equation}\label{dsnvnndvjfdljkaf}
	\left\|\left(B^\om_{a}\right)^{(n)}\right\|_{A^2_\om}^2
	\asymp|a|^{2n}\sum_{k=0}^\infty\frac{(k+1)^{2n}\om_{2k+1}}{\om_{2(k+n)+1}^2}|a|^{2k},\quad a\in\D,
	\end{equation}
for each fixed radial weight $\omega$. If now $\om\in\DD$, then
	\begin{equation}\label{edjlekjvbleaj}
	\begin{split}
	\left\|\left(B^\om_{a}\right)^{(n)}\right\|_{A^2_\om}^2
	&\asymp|a|^{2n}\sum_{k=0}^\infty\frac{(k+1)^{2n}}{\om_{2k+1}}|a|^{2k}
	\asymp|a|^{2n}\left(\int_1^\infty\frac{x^{2n}}{\om_{2x}}|a|^{2x}\,dx+1\right)\\
	&\le|a|^{2n}\left(\int_0^1\frac{|a|^{\frac{1}{1-t}}}{\om_{\frac{1}{1-t}}}\frac{dt}{(1-t)^{2(n+1)}}+1\right),\quad a\in\D.
	\end{split}
	\end{equation}
Clearly,
	\begin{equation}\label{edjlekjvbleaj2}
	\int_0^{|a|}\frac{|a|^{\frac{1}{1-t}}}{\om_{\frac{1}{1-t}}}\frac{dt}{(1-t)^{2(n+1)}}
	\lesssim\frac{1}{\om_{\frac{1}{1-|a|}}(1-|a|)^{2n+1}},\quad a\in\D.
	\end{equation}
To establish the same upper estimate for the integral over $(|a|,1)$, we denote $a_j=1-\frac{1-|a|}{2^j}$ for all $j\in\N_0$. Since $\omega$ is doubling, there exists $\beta=\beta(\om)>0$ such that $x\mapsto x^\beta\om_x$ is essentially increasing on $[1,\infty)$ by \cite[Lemma~2.1(ii)(vi)]{PelSum14}. Therefore
	\begin{equation}\label{edjlekjvbleaj3}
	\begin{split}
	&\int_{|a|}^1\frac{|a|^{\frac{1}{1-t}}}{\om_{\frac{1}{1-t}}}\frac{dt}{(1-t)^{2(n+1)}}\\
	&\quad\lesssim\frac{(1-|a|)^\beta}{\om_{\frac{1}{1-|a|}}}\sum_{j=0}^\infty
	\int_{a_j}^{a_{j+1}}\frac{|a|^{\frac{1}{1-t}}}{(1-t)^{2(n+1)+\beta}}dt\\
	&\quad\le\frac{(1-|a|)^\beta}{\om_{\frac{1}{1-|a|}}}\sum_{j=0}^\infty\frac{|a|^{\frac{1}{1-a_j}}}{(1-a_{j+1})^{2(n+1)+\beta}}(a_{j+1}-a_j)\\
	&\quad\lesssim\frac{1}{\om_{\frac{1}{1-|a|}}(1-|a|)^{2n+1}},\quad a\in\D.
	\end{split}
	\end{equation}
By combining \eqref{edjlekjvbleaj}, \eqref{edjlekjvbleaj2} and \eqref{edjlekjvbleaj3} we obtain (iv).

It remains to show that (iv) implies $\om\in\DD$. By (iv) and \eqref{dsnvnndvjfdljkaf} we have
	$$
	\frac{|a|^{2n}}{\om_{\frac{1}{1-|a|}}(1-|a|)^{2n+1}}
	\gtrsim\left\|\left(B^\om_{a}\right)^{(n)}\right\|_{A^2_\om}^2
	\gtrsim|a|^{2n}\int_1^\infty\frac{x^{2n}}{\om_{2x}}|a|^{2x}\,dx,\quad a\in\D.
	$$
The argument employed in \eqref{lkjflerjlqkerjnv} now yields a suitable lower estimate for the last integral allowing us to deduce $\om\in\DD$. This completes the proof of the theorem.\hfill$\Box$

\medskip

We note here that another natural way to deduce (iv) from the other (equivalent) statements is to apply the pointwise estimate (i) inside the integral of the $A^2_\om$-norm of $\left(B^\om_{a}\right)^{(n)}$, and then estimate this upper bound further to reach the statement (iv). This second approach sounds rough but it works fine. In fact, we will use it efficiently in the proof of Corollary~\ref{Corollary:Integrated kernel estimates} below in a more general setting in order to obtain sharp integral estimates for modified Bergman kernels.

\section{Proofs of Corollaries}\label{sec3}

\noindent\emph{Proof of Corollary~\ref{Corollary:Integrated kernel estimates}.}
Write $\overline{a}z=te^{i\theta}$ and $\nu(r)=\om(\sqrt r)$ for all $0\le r<1$ as in the proof of Theorem~\ref{Theorem:off-diagonal kernel}. Then the said theorem together with the estimate $|1-te^{i\theta}|^2\asymp(1-t)^2+t\theta^2$ implies
	\begin{equation*}
	\begin{split}
	I(a,r)
	&=\int_0^{2\pi}\left|(1-\overline{a}re^{i\varphi})^\beta\left(B^\om_a\right)^{(n)}(re^{i\varphi})\right|^p\,d\varphi\\
	&\lesssim|a|^{np}\int_0^{2\pi}\frac{\kappa(te^{i\theta})^{p(n+1-\beta)}}{\nu_{\kappa(te^{i\theta})}^p}\,d\theta\\
	&\asymp|a|^{np}\int_0^{2\pi}\frac{d\theta}{\widehat{\nu}\left(1-\frac{|1-te^{i\theta}|}{2}\right)^p|1-te^{i\theta}|^{p(n+1-\beta)}}\\
	&\asymp|a|^{np}\left(\int_0^{1-t}+\int_{1-t}^1+\int_1^\pi\right)
	\frac{((1-t)^2+t\theta^2)^{-\frac{p}{2}(n+1-\beta)}}{\widehat{\nu}\left(1-\frac{((1-t)^2+t\theta^2)^\frac12}{4}\right)^p}\,d\theta\\
	&=|a|^{np}\left(I_1(a,r)+I_2(a,r)+I_3(a,r)\right),
	\end{split}
	\end{equation*}
where
	\begin{equation*}
	\begin{split}
	I_1(a,r)
	&\lesssim\frac1{\widehat{\nu}(t)^p(1-t)^{p(n+1-\beta)}}\int_0^{1-t}d\theta\\
	&=\frac1{\widehat{\nu}(t)^p(1-t)^{p(n+1-\beta)-1}},\quad 0<t<1,
	\end{split}
	\end{equation*}
and
	\begin{equation*}
	\begin{split}
	I_2(a,r)
	&\lesssim\int_{1-t}^1\frac{d\theta}{\widehat{\nu}(1-\theta)^p\theta^{p(n+1-\beta)}}+1\\
	&=\int_0^t\frac{ds}{\widehat{\nu}(s)^p(1-s)^{p(n+1-\beta)}}+1,\quad \frac12\le t<1.
	\end{split}
	\end{equation*}
Since $\nu=\om(\sqrt\cdot)$ and $\om\in\DD$ by the hypothesis, \cite[Lemma~2.1(ii)]{PelSum14} implies the existence of $C=C(\om)\ge 1$ and $\gamma=\gamma(\om)>0$ such that
    \begin{equation*}
    \begin{split}
    \widehat{\nu}(r)\le C\left(\frac{1-r}{1-t}\right)^{\gamma}\widehat{\nu}(t),\quad 0\le r\le t<1.
    \end{split}
    \end{equation*}
A direct application of this inequality shows that
	$$
	\frac1{\widehat{\nu}(t)^p(1-t)^{p(n+1-\beta)-1}}
	\lesssim1+\int_0^t\frac{ds}{\widehat{\nu}(s)^p(1-s)^{p(n+1-\beta)}},\quad 0<t<1.
	$$
Further, since $I_3(a,r)$ is uniformly bounded, the first statement is proved for $|\overline{a}z|\ge\frac12$. The case $|\overline{a}z|<\frac12$ follows from the proof above with minor modifications concerning $I_2$. The second statement is an immediate consequence of the first one and Fubini's theorem.\hfill$\Box$

\medskip

\noindent\emph{Proof of Corollary~\ref{Corollary:Maximal function of Berezin transform}.}
First observe that
	$$
	T_{\delta,\om}(f)(z)
	\asymp\om(S(z))^\delta\int_\D|f(\zeta)||B^\om_z(\zeta)|^{1+\delta}\om(\zeta)\,dA(\zeta),\quad z\in\D,\quad f\in L^1_\om,
	$$
by Theorem~\ref{Theorem:off-diagonal kernel}(iv) and the note following it. We split $\D$ into pieces based on the point $z$ as follows. For $z\in\D\setminus\{0\}$ and $K>1$ such that $K(1-|z|)<1$ we write $S_K(z)=S\left((1-K(1-|z|)e^{i\arg z}\right)$. If $K(1-|z|)\ge 1$, we set $S_K(z)=\D$. Then $|1-\overline{z}\zeta|\asymp1-|z|$, if $\zeta\in S(z)$, and $|1-\overline{z}\zeta|\asymp K^k(1-|z|)$, for $\zeta\in S_{K^k}(z)\setminus S_{K^{k-1}}(z)$ and $k\in\N$. Let $N\in\N$ be the smallest natural number such that $K^N(1-|z|)\ge1$. Theorem~\ref{Theorem:off-diagonal kernel} implies
	\begin{equation*}
	\begin{split}
    T_{\delta,\om}(f)(z)
    &\asymp\om(S(z))^\delta\sum_{k=1}^{N}\int_{S_{K^k}(z)\setminus S_{K^{k-1}}(z)}|f(\zeta)||B^\om_z(\zeta)|^{1+\delta}\om(\zeta)\,dA(\zeta)\\
    &\quad +\om(S(z))^\delta\int_{S(z)}|f(\z)||B^{\om}_z(\z)|^{1+\delta}\om(\z)\,dA(\z)\\
    &\lesssim \om(S(z))^\delta\sum_{k=1}^{N}\frac{1}{\om(S_{K^{k}}(z))^{1+\delta}}\int_{S_{K^{k}}(z)}|f(\z)|\om(\z)\,dA(\z)\\
    &\quad + \frac{1}{\om(S(z))}\int_{S(z)}|f(\z)|\om(\z)\,dA(\z)\\
    &\lesssim M_{\om}(f)(z)\left(\sum_{k=1}^{N}\frac{1}{K^{\delta k}}+1\right)
    \lesssim M_{\om}(f)(z), \quad z \in \D.
	\end{split}
	\end{equation*}
The statement then follows by a simple geometric argument.
\hfill$\Box$

\medskip

\noindent\emph{Proof of Corollary~\ref{Corollary:B-boundedness}.}
The statements (i), (iii) and (iv) are equivalent by \cite{PelRatMathAnn}, and Corollary~\ref{Corollary:Maximal function of Berezin transform} shows that (i) implies (ii). Assume now (ii), and define $f_a(z)=\left(\frac{1-|a|}{1-\overline{a}z}\right)^\eta$, where $a\in\D$ and $\eta>0$. By \cite[Lemma~2.1(vii)]{PelSum14} we may choose $\eta=\eta(\om)>0$ sufficiently large such that $\|f_a\|_{L^p_\om}^p\asymp\om(S(a))$ for all $a\in\D$. Write $a_\e=(1-\e(1-|a|))e^{i\arg a}$ for each
$\e\in(0,1]$ and $a\in\D\setminus\{0\}$. Then, by choosing $\e=\e(\om)>0$ sufficiently small, \cite[Lemma~7]{PRS2} yields
	\begin{equation*}
	\begin{split}
	&\om(S(a))
	\asymp\|f_a\|_{L^p_\om}^p\\
	&\gtrsim\int_\D\left(\om(S(z))^\delta\int_\D\left(\frac{1-|a|}{|1-\overline{a}\zeta|}\right)^\frac{\eta}{\alpha}|B^\om_z(\zeta)|^{1+\delta}\om(\zeta)\,dA(\zeta)\right)^{p\alpha}\,d\mu(z)\\
	&\gtrsim\int_{S(a)}\left(\om(S(z))^\delta\int_{S_\e(z)}|B^\om_z(\zeta)|^{1+\delta}\om(\zeta)\,dA(\zeta)\right)^{p\alpha}\,d\mu(z)\\
	&\gtrsim\int_{S(a)}\left(\om(S(z))^\delta\frac1{\om(S(z))^{1+\delta}}\om(S_\e(z))\right)\,d\mu(z)
	\asymp\mu(S(a)),\quad a\in\D.
	\end{split}
	\end{equation*}
Therefore (iv) is satisfied and the proof is complete.
\hfill$\Box$

\section{Further discussion}\label{sec4}

In this section we discuss several results that the method of proof of Theorem~\ref{Theorem:off-diagonal kernel} allows us to establish quite effortlessly. We begin with fractional derivatives. Their roots can be traced back to private correspondence between G.~l'H\^opital and G.~W.~Leibniz in 1695, the Euler transform, proposed as early as 1769 by L.~Euler, and J.~Liouville's studies on fractional calculus from 1832. However, the land marking work by G.~H.~Hardy and J.~E.~Littlewood~\cite{HL1932} from 1932 and the subsequent studies by T.~M. Flett~\cite{Flett1,Flett2} some decades later can be considered as the starting point of the modern study of fractional derivatives.

Our first result on this direction concerns the one-weight fractional derivative
	$$
	D^{\nu}(f)(z)=\sum_{n=0}^{\infty}\frac{\widehat{f}(n)}{\nu_{2n+1}}z^n,\quad z\in\D,\quad f\in\mathcal{H}(\D),
	$$
which is essentially an extension of the (complex) Riemann-Liouville fractional derivative. The method of proof of our main result works equally well for $D^\nu$ as for the standard derivatives, provided $\nu\in\DD$, and yields the following result.

\begin{theorem}\label{Cor:Fracresult}
Let $\om$ be a radial weight and $\nu\in\DD$. Then the following statements are equivalent:
		\begin{itemize}
    \item[(i)]$\left|D^{\nu}(B^{\om}_a)(z)\right|
		\lesssim\frac{\kappa(a,z)}{\om_{\kappa(a,z)}\nu_{\kappa(a,z)}}, \quad \kappa(a,z)=\frac{2}{|1-\overline{a}z|}, \quad a,z\in \D;$
    \item[(ii)]$\left\|D^{\nu}(B^{\om}_a)\right\|_{H^{\infty}}
		\lesssim\frac{1}{\om_{\frac{1}{1-|a|}}\nu_{\frac{1}{1-|a|}}(1-|a|)},\quad a\in\D;$
    \item[(iii)]$\left\|D^{\nu}(B^{\om}_a)\right\|_{\mathcal{B}}
		\lesssim\frac{1}{\om_{\frac{1}{1-|a|}}\nu_{\frac{1}{1-|a|}}(1-|a|)},\quad a\in\D;$
    \item[(iv)]$\left\|D^{\nu}(B^{\om}_a)\right\|_{A^2_{\omega}}^2
		\lesssim\frac{1}{\om_{\frac{1}{1-|a|}}\nu_{\frac{1}{1-|a|}}^2(1-|a|)},\quad a\in\D;$
    \item[(v)]$\om\in\DD$.
    \end{itemize}
\end{theorem}

The proof of Theorem~\ref{Theorem:off-diagonal kernel} is rich in details, and since, at the end of the day, one can fairly say that the method of proof there is relatively simple, we trust that its extension to one-weight fractional derivatives given above is easily reachable without further guidance. Therefore here, as well as along this section, we omit most of the details of the proofs. Before proceeding further, we invite the reader to consult \cite{BMNP,MPR2024,Pel-Rosa-2022} regarding to recent studies related to $D^\nu$.

The two-weight fractional derivative 	
	$$
	R^{\tau,\nu}(f)(z)=\sum_{n=0}^\infty\frac{\tau_{2n+1}}{\nu_{2n+1}}\widehat{f}(n)z^n,\quad z\in\D,
	$$
induced by radial weights $\tau$ and $\nu$, was defined recently in \cite{P2020}. This derivative is essentially a generalization of the operator studied about two decades ago in \cite{ZZ2008,Zhu2,Zhu3}. Although it is not immediate from the definitions, it is known that the set of operators $R^{\tau,\nu}$ contains all one-weight fractional derivatives \cite{P2020}. Minor adjustments in the first part of the proof of Theorem~\ref{Theorem:off-diagonal kernel} give us
	\begin{equation}\label{IFF}
	|R^{\tau,\nu}(B^\omega_a)(z)|\lesssim\int_1^{\kappa(a,z)}\frac{\tau_t}{\nu_t\om_t}\,dt+1,\quad a,z\in\D,
	\end{equation}
provided all weights involved belong to $\DD$. The appearance of an integral on the right is fully expected as $R^{\tau,\nu}$ may represent an integral of high order, and hence $R^{\tau,\nu}(B^\omega_a)$ may constitute a uniformly bounded family of bounded analytic functions. If we assume that this is not the case, and further that the integral on the right behaves well in the sense that
	$$
	\int_1^{\kappa}\frac{\tau_t}{\nu_t\om_t}\,dt
	\lesssim\frac{\tau_\kappa}{\nu_\kappa\om_\kappa}\kappa,\quad 1\le\kappa<\infty,
	$$
then it turns out that \eqref{IFF} ensures $\om\in\DD$ under the hypothesis that both $\tau$ and $\nu$ belong to~$\DD$. Therefore, by taking into account the appropriate variants of the statements (ii)-(iv) in Theorem~\ref{Cor:Fracresult} for the two-weight case, we can achieve an equally covering ``if and only if''-result also in this two-weight setting. Before going further, we mention \cite{DGRW,PRW} as references on recent studies related to the two-weight fractional derivatives. We also emphasize here that only a quick look on these works shows the apparent usefulness of $R^{\tau,\nu}$ in the operator theory related to the Bergman spaces induced by radial (doubling) weights.

The two-weight fractional derivative Bergman kernel estimate \eqref{IFF} can be easily used to obtain sharp pointwise upper estimates for the reproducing kernels of Dirichlet spaces. To make this rigorous, let~$\DDD^2_\omega$ denote the weighted Dirichlet space endowed with the norm
	$$
	\|f\|_{\DDD^2_\omega}=\|f'\|_{A^2_\omega}+|f(0)|.
	$$
Normalized monomials give the basis $\left\{\frac{z^n}{n\sqrt{2}\sqrt{\omega_{2n-1}}}\right\}_{n=1}^\infty\cup\{1\}$ and show that the reproducing kernel of $\mathcal D^2_\om$ has the representation
	\begin{equation*}
	D^\omega_a(z)
	=1+\frac{\overline{a}z}2\sum_{n=0}^\infty\frac{(\overline{a}z)^n}{(n+1)^2\omega_{2n+1}}
	=1+\overline{a}zK^\omega_a(z),
	\end{equation*}
where
	$
	K^\omega_a=R^{\tau,\nu}(B^\omega_a)
	$
for $\nu\equiv1$ and $\tau=2\left(\log\frac1{|\cdot|}\right)^2$. To establish this last identity, one just has to verify $\frac{\tau_{2n+1}}{\nu_{2n+1}}=\frac1{(n+1)^2}$ for all $n\in\N\cup\{0\}$ for our choices of weights. It then follows from~\eqref{IFF} that
	\begin{equation}\label{DDDDDDDDD}
	|D^\omega_a(z)-1|\lesssim|\overline{a}z|\left(\int_1^{\kappa(a,z)}\frac{dt}{t^2\om_t}+1\right),\quad a,z\in\D.
	\end{equation}
This argument applies with appropriate modification to higher-order derivatives as well. The estimate \eqref{DDDDDDDDD} is sharp pointwise up to a multiplicative constant as is seen by considering the kernels induced by the standard radial weights. In view of our main result, \eqref{DDDDDDDDD} comes by no means as a surprise as the reproducing kernels of the Dirichlet space $\mathcal D^2_\om$ are, roughly speaking, the second primitives of the corresponding Bergman kernels. We finish our discussion on fractional derivatives by noting that, with all the techniques used above in hand, one may now establish integrated fractional derivative and Dirichlet reproducing kernel estimates in the spirit of Corollary~\ref{Corollary:Integrated kernel estimates} by following the reasoning used in the proof of the said corollary. Since these arguments do not reveal us anything essentially new, we do not get into details and leave them for interested readers. To this end, we just mention that, for example, the upper bounds for the integrated Dirichlet reproducing kernel estimates related to certain regular weights discovered in \cite[Theorem~4]{PR2016/2} can be easily recovered by our arguments.

Next, we note that the conclusions of Theorem~\ref{Theorem:off-diagonal kernel} extend to the harmonic case. The harmonic Bergman space $h^2_{\om}$ consists of the harmonic functions in $L^2_\om$, and its reproducing kernel is
		\begin{equation*}
		\begin{split}
		b^{\om}_a(z)
		=\frac12\sum_{n=0}^{\infty}\frac{(\cona z)^n}{\om_{2n+1}}
		+\frac12\sum_{n=1}^{\infty}\frac{(a\conz)^n}{\om_{2n+1}}
		=B^{\om}_a(z)+\overline{B^{\om}_a(z)-\frac{1}{2\om_1}},\quad a,z\in\D.
		\end{split}
		\end{equation*}
Further, the harmonic Bloch space $\mathcal{B}_h$ is defined by the condition
		\begin{equation*}
    \left\|f\right\|_{\mathcal{B}_h}
		=\sup_{z\in\D}\left(|g'(z)|+|h'(z)|\right)(1-|z|^2)+|f(0)|<\infty,\quad f=g+\overline{h},\quad g,h\in\mathcal{H}(\D).
		\end{equation*}
The harmonic counterpart of Theorem~\ref{Theorem:off-diagonal kernel} readily follows from the proof of the analytic case, and is stated as follows.

\begin{theorem}\label{harmonic case}
Let $\om$ be a radial weight. Then the following statements are equivalent:
		\begin{itemize}
    \item[(i)] $|b^{\om}_a(z)|\lesssim\frac{\kappa(a,z)}{\om_{\kappa(a,z)}}, \quad \kappa(a,z)=\frac{2}{|1-\overline{a}z|}, \quad a,z\in \D;$
    \item[(ii)] $\left\|b^{\om}_a\right\|_{h^{\infty}} \lesssim \frac{1}{\om_{\frac{1}{1-|a|}}(1-|a|)}, \quad a \in \D;$
    \item[(iii)] $\left\|b^{\om}_a\right\|_{\mathcal{B}_h} \lesssim \frac{1}{\om_{\frac{1}{1-|a|}}(1-|a|)}, \quad a \in \D;$
    \item[(iv)]  $\left\|b^{\om}_a\right\|_{h^2_{\omega}}^2 \lesssim \frac{1}{\om_{\frac{1}{1-|a|}}(1-|a|)}, \quad a \in \D;$
    \item[(v)] $\om \in \DD$.
    \end{itemize}
\end{theorem}

Our last further observation concerns the underlying dimension. Let $\mathbb{B}_n=\{z \in \C^n: |z|<1\}$ denote the unit ball of the $n$-dimensional complex plane $\C^n$. For a radial weight $\om: \mathbb{B}_n \to [0,\infty)$, the reproducing kernel of $A^2_{\om}$ admits the series expansion
		\begin{equation*}
    B^{\om}_{a}(z)=\frac{1}{2n!}\sum_{k=0}^{\infty}\frac{(k+n-1)!}{k!\om_{2k+2n-1}}\langle z,a\rangle^k, \quad z,a \in \mathbb{B}_n,
		\end{equation*}
where $\langle\cdot,\cdot\rangle$ stands for the usual inner product on $\C^n$~\cite[p.~9--14]{Zhu2}.  The Bloch space $\mathcal{B}$ consists of all holomorphic functions on $\mathbb{B}_n$ satisfying
		\begin{equation*}
    \left\|f\right\|_{\mathcal{B}}
		=\sup_{z\in\mathbb{B}_n}|R(f)(z)|(1-|z|^2)+|f(0)|<\infty,
		\end{equation*}
where $R(f)$ is the radial derivative, defined as $R(f)(z)=\sum_{j=1}^{n}z_j\frac{\partial}{\partial z_j}f(z)$ for $z=(z_1,\dots,z_n) \in \mathbb{B}_n$. For a radial weight $\om$ in $\mathbb{B}_n$, we write $\om\in\DD$ if there exists a constant $C=C(\om)>0$ such that
		$$
		\om_x\le C\om_{2x}, \quad2n-1\le x<\infty.
		$$
An analogue of Theorem~\ref{Theorem:off-diagonal kernel} in the $n$-dimensional setting is given by the following result.

\begin{theorem}\label{unit ball}
Let $n\in\N$ and $\om$ be a radial weight in $\mathbb{B}_n$. Then the following statements are equivalent:
		\begin{itemize}
		\item[(i)] $|B^{\om}_a(z)|\lesssim\frac{\kappa(a,z)^n}{\om_{\kappa(a,z)}},
		\quad\kappa(a,z)=\frac{4n-2}{|1-\langle z,a\rangle|}, \quad a,z\in \mathbb{B}_n;$
		\item[(ii)] $\left\|B^{\om}_a\right\|_{H^{\infty}}
		\lesssim \frac{1}{\om_{\frac{2n-1}{1-|a|}}(1-|a|)^n}, \quad a \in \mathbb{B}_n;$
		\item[(iii)] $\left\|B^{\om}_a\right\|_{\mathcal{B}}
		\lesssim \frac{1}{\om_{\frac{2n-1}{1-|a|}}(1-|a|)^n}, \quad a \in \mathbb{B}_n;$
    \item[(iv)] $\left\|B^{\om}_a\right\|_{A^2_{\omega}}^2
		\lesssim \frac{1}{\om_{\frac{2n-1}{1-|a|}}(1-|a|)^n}, \quad a \in \mathbb{B}_n;$
    \item[(v)] $\om\in\DD$.
		\end{itemize}
\end{theorem}

The dimension-related factors $4n-2$ and $2n-1$ appearing in the statement are needed because the moment $\om_x$ is required to be well defined only for $x\ge 2n-1$. Now that Theorem~\ref{unit ball} is established, one can certainly reach out for higher dimensional analogues of Corollaries~\ref{Corollary:Integrated kernel estimates},~\ref{Corollary:Maximal function of Berezin transform} and \ref{Corollary:B-boundedness} or pointwise estimates for the fractional derivatives of the kernel functions, but in order to avoid unnecessary repetition we rest our case here.

\end{document}